\documentclass[dvips,twoside] {article}
\usepackage[a5paper,layoutwidth=148mm,layoutheight=7.77817in,
outer=10mm,inner=35bp,tmargin=10mm,bmargin=10mm,footskip=15bp,
]{geometry}
\usepackage[utf8]{inputenc}
\usepackage[T1]{fontenc}
\usepackage{lmodern}
\usepackage{amsfonts}
\usepackage{amssymb}
\usepackage{amsthm}
\usepackage[bookmarksopen=false,breaklinks=true,%
backref=page,pagebackref=true,plainpages=false,%
hyperindex=true,pdfstartview=FitH,colorlinks=true,%
pdfpagelabels=true,linkcolor=blue,%
citecolor=red,urlcolor=red,hypertexnames=false%
]%
{hyperref}

\usepackage{doi}
\usepackage{mathtools}
\usepackage{natbib}
\bibpunct (),.{},
\usepackage[british]{babel}
\usepackage{xspace}

\usepackage{etoolbox}
\makeatletter\def\th@plain{\slshape}
\patchcmd{\th@remark}{\itshape}{\slshape}{}{}\makeatother

\newcounter{bidon}
\newcommand{\rdb}{\refstepcounter{bidon}}

\theoremstyle{plain}
\newtheorem{theorem}{Theorem}
\newtheorem*{theoremcinqtrois}{Theorem 6.3}
\newtheorem{lemma}[theorem]{Lemma}
\newtheorem{corollary}[theorem]{Corollary}

\theoremstyle{definition}
\newtheorem{remark}[theorem]{Remark}
\newcommand\ndsp{\textstyle}

\renewcommand \geq{\geqslant}

\newcommand\eti{^\times}
\newcommand \Ati {\gA^{\!\times}}
\newcommand \formu[1]{{\arraycolsep2pt\begin{array}{rcl} #1 \end{array}}}

\newcommand \aqo[2] {#1\sur{\gen{#2}}\!}
\newcommand \gen[1] {\left\langle{#1}\right\rangle}
\newcommand \so[1] {\left\{\,{#1}\, \right\}}
\newcommand \sotq[2]{\so{\,#1\,\vert\,#2\,}}
\newcommand \sur[1] {\!\left/#1\right.}
\renewcommand \geq{\geqslant}

\newcommand \som {\sum\nolimits}
\newcommand\Aq[2]{#1_{[#2]}}
\newcommand\Aqj[2]{#1_{\{#2\}}}

\newcommand \eoe {\hbox{}\nobreak\hfill
\vrule width .5em height .5em depth 0mm \par \smallskip}

\newcommand \NN{\mathbb {N}}
\newcommand \gA {\mathbf{A}}
\newcommand \gB {\mathbf{B}}
\newcommand \gC {\mathbf{C}}
\newcommand \gK {\mathbf{K}}
\newcommand \rhe {^{\mathrm{h}}}
\newcommand \rHe {^{\mathrm{H}}}
\newcommand \KHe {\gK\rHe}
\newcommand \gL {\mathbf{L}}
\newcommand \gV {\mathbf{V}}
\newcommand \Vhe {\gV\rhe}
\newcommand \VHe {\gV\rHe}
\newcommand \fmhe {\fm\rhe}
\newcommand \fmHe {\fm\rHe}
\newcommand \mhe {\fmhe}
\newcommand \mHe {\fmHe}
\newcommand \Ahe {\gA\!\rhe}

\newcommand \KV {(\gK,\gV)}
\newcommand \gW {\mathbf{W}}
\newcommand \Kac {{\gK^\mathrm{ac}}}
\newcommand \Vac {{\gV^\mathrm{ac}}}
\newcommand \Ann {\mathrm{Ann}}

\renewcommand \deg {\MA{\mathrm{deg}}}

\newcommand \Frac {\MA{\mathrm{Frac}}}

\newcommand \Ker {\MA{\mathrm{Ker}}}
\newcommand \pgcd {\MA{\mathrm{pgcd}}}

\newcommand \Rad {\MA{\mathrm{Rad}}}

\newcommand \Res {\mathrm{Res}}

\newcommand \tr {\MA{\mathrm{tr}}}

\newcommand\MA[1]{\mathop{#1}\nolimits}

\newcommand\fm{\mathfrak{m}}

\newcommand \AX {\gA[X]}
\newcommand \Ax {\gA[x]}

\newcommand \KX {\gK[X]}
\newcommand \Kx {\gK[x]}
\newcommand \VX {\gV[X]}

\newcommand \VT {\gV[T]}

\patchcmd{\sectionmark}{\MakeUppercase}{}{}{}
\usepackage{authblk}

\begin{document}%\raggedright\pagestyle{empty}

\author[$\dag$]{María Emilia Alonso García}
\author[$\ddag$]{Henri Lombardi}
\author[$\ddag$]{Stefan Neuwirth}
\affil[$\dag$]{Departamento de Álgebra, Geometría y Topología, Facultad de Ciencias Matemáticas, Universidad Complutense Madrid, Spain, \url{mariemi@mat.ucm.es}.}
\affil[$\ddag$]{Université de Franche-Comté, CNRS, UMR 6623, LmB, 25000 Besançon, France, \url{henri.lombardi@univ-fcomte.fr}, \url{stefan.neuwirth@univ-fcomte.fr}.}

\title{Note on the coincidence of two henselisations}
\date{}
\maketitle%\thispagestyle{empty}

\begin{abstract} 
We compare two henselisations of a residually discrete valuation domain. Our constructive proof that a certain natural morphism is an isomorphism is also a proof in classical mathematics. Although this isomorphism is implicitly accepted as obvious in the literature, it seems that no proof was previously available.  
\end{abstract}

This note is written in Bishop's style of constructive mathematics \citep*{Bi67,BB85,BR1987,MRR,CACM,CCAPM}.

\medskip \noindent 
{\bf Terminology, notation.} 

\smallskip \noindent This paragraph clarifies the usual constructive terminology for a number of classical notions. From a constructive point of view, the development of good definitions\footnote{A ``good definition'' must have a clear constructive meaning, apply to cases usually dealt with in current literature, and be equivalent in classical mathematics to the most common definition. Naturally, this is not always a simple matter. Several ``good'' constructively inequivalent definitions can apply to a single classical notion and it is sometimes useful to consider them.} is a decisive part of the work which consists in extracting the constructive content of a concrete result obtained by means of a classical proof.

\smallskip \noindent A ring is \textsl{normal} when every principal ideal is integrally closed. 
In classical mathematics, this constructive definition is shown to be equivalent to the more usual definition according to which the ring becomes an integrally closed domain by localisation at any prime ideal. 

\smallskip \noindent A ring is \textsl{connected} if every idempotent is equal to~$0$ 
 or~$1$.

 \smallskip \noindent Let $\gK$ be a discrete field. A $\gK$-vector space~$E$ is \textsl{finite} over~$\gK$ (respectively \textsl{strictly} finite over~$\gK$) if it is a $\gK$-vector space of finite type (respectively if we know a finite basis of~$E$ over~$\gK$). More generally, if $\gA$ is a ring, an $\gA$-module~$E$ is \textsl{finite} over~$\gA$ if it is an $\gA$-module of finite type, and \textsl{strictly} finite over~$\gA$ if it is an $\gA$-projective module of finite type. 

 \smallskip \noindent A $\gK$-algebra~$\gL$ is \textsl{étale} if it is strictly finite and if the trace form is nondegenerate; equivalently, it is strictly finite and generated by separable elements over~$\gK$. More generally, if $\gA$ is a ring, an $\gA$-algebra~$\gB$ is \textsl{strictly} étale if it is strictly finite and if the trace form is nondegenerate. See for example \citet[Definitions~VI-1.1 and~VI-5.1, Theorems~VI-1.4, VI-1.7, VI-5.4, and~VI-5.5]{CACM}.

\smallskip \noindent The \textsl{Jacobson radical} of a ring~$\gA$, denoted by $\Rad(\gA)$, is the ideal formed by the~$x\in\gA$ such that $1+x\gA\subseteq \Ati$ (the multiplicative group of units).

\smallskip \noindent A \textsl{pp-ring} (or \textsl{quasi-integral ring}) is a ring in which the annihilator of any element is generated by an idempotent. A quasi-integral ring is normal if, and only if, it is integrally closed in its total ring of fractions. An integral domain is none other than a connected quasi-integral ring.

\smallskip \noindent Let $\gA$~be a \textsl{local ring}, i.e., for any~$x\in\gA$, $x$ or~$1-x$ is a unit. Its Jacobson radical $\Rad(\gA)$, also denoted by~$\fm$ or~$\fm_\gA$, is equal to the ideal formed by the nonunits. It is called a \textsl{Heyting field} if $\Rad(\gA)=\{0\}$. 
The residual field of~$\gA$  is the quotient ring~$\gA/\Rad(\gA)$; $\gA$~is \textsl{residually discrete} when the residual field is a discrete field; this is equivalent to having a test for ``$x\in\Ati$ or $x\in\Rad(\gA)$?'' in~$\gA$.

\smallskip\noindent A morphism in the category of local rings is a ring morphism $\varphi\colon\gA\to\gC$ which reflects units, i.e.\ such that
$\varphi(x)\in\gC\eti$ implies $x\in\Ati$. Such a morphism is sometimes called local, to distinguish it from mere ring morphisms.  
In the category of residually discrete local rings, the morphism $\varphi\colon\gA\to\gC$ is said to define~$\gC$ as an \textsl{$\gA$-local-algebra}.\footnote{``$\gA$-local-algebra'' must be read as a single word: it is an abbreviation in which the word ``local'' refers to both the object~$\gC$ and the morphism $\varphi$.}

\smallskip \noindent A \textsl{Hensel code} in a residually discrete local ring $(\gA,\fm)$ is a pair
$(f,a)\in \AX\times \gA$ with~$f$ monic, $f(a)\in\fm$ and $f'(a)\in\Ati$.
 
\smallskip \noindent For a residually discrete local ring~$\gA$ and a Hensel code $(f,a)$, \textsl{a Hensel zero of code $(f,a)$} in an $\gA$-local-algebra $\varphi\colon\gA\to\gC$ is an element $\alpha\in\gC$ which verifies $f(\alpha)=0$ and 
$\alpha-\varphi(a)\in\Rad(\gC)$ (we then have $f'(\alpha)\in\gC\eti$). This zero is necessarily unique. A \textsl{special polynomial} is a polynomial 
$g(X)=X^n-X^{n-1}+a_{n-2}X^{n-2}+\dots+a_0%\sum_{k=0}^{n-2} a_kX^k
$ with the~$a_k$'s in~$\fm$; in this case, $(g,1)$ is a Hensel code whose Hensel zero is \textsl{the special zero} of~\(g\). A polynomial $f\in\AX$ such that $(f,0)$ is a Hensel code is called a \textsl{Hensel polynomial}; in this case a Hensel zero of code $(f,0)$ is simply called a Hensel zero of~$f$ in~$\gC$.

\smallskip \noindent A residually discrete local ring~$\gA$ is said to be \textsl{henselian} if every Hensel polynomial has a Hensel zero %(of code $(f,0)$)
in~$\gA$.

\smallskip \noindent A \textsl{discrete valued field}~$\KV$ is a discrete field~$\gK$ with a subring~$\gV$ which verifies: $xy=1\in\gK \implies x\in\gV$ or $y\in\gV$. 
In this case $\gV$ is a normal local ring. We also require that $\gV$ be residually discrete. This is equivalent to requiring that we have a test for ``$\exists x\in\gV$ such that $a=xb$?'', where $a,b\in\gK$.

\smallskip \noindent Let $\gA$ be an arbitrary commutative ring and $f\in\AX$ a monic polynomial; we shall use the notation \fbox{$\Aq \gA f=\aqo\AX f$} and~\fbox{$\Aqj \gA f=\Aq \gA f[1/f']$}.

\smallskip \citet[Theorem 6.3]{CL2016b} state the following helpful nontrivial result.

\begin{theoremcinqtrois}
  If $\gA$ is a normal ring and $f\in\AX$ a monic polynomial, then $\gA_{\{ f\}}$ is also a normal ring.
\end{theoremcinqtrois}

\rdb\addcontentsline{toc}{subsection}{Presentation of the problem} \subsection*{Presentation of the problem}
 We answer a natural question that arises from the articles \citealt{KL00} and \citealt*{ALP08}. 
% We use the constructive versions (adopted in these two articles) of the usual classical definitions.

 \smallskip The henselisation of a residually discrete local ring $(\gA,\fm)$, denoted by $(\Ahe,\fmhe)$, is constructed in \citealt*{ALP08}. It solves the universal problem for the full subcategory of henselian residually discrete local rings in the category of residually discrete local rings. In other words the construction is functorial and the functor ``henselisation'' is the left adjoint functor to the inclusion functor.
 
\smallskip The henselisation of a discrete valued field $(\gK,\gV)$, denoted by $(\KHe,\VHe)$, is constructed in \citealt{KL00}. It solves the universal problem for the full subcategory of henselian discrete valued fields in the category of discrete valued fields. 

Since $\gV$ is a residually discrete local ring, it is natural to compare the two henselisations~$\Vhe$ and~$\VHe$. 

This problem does not seem to be addressed in the classical literature, perhaps because the answer seems obvious. But it has not been obvious to us, hence this note.
It also answers the problem from the point of view of classical mathematics, since our constructive definitions are equivalent to the usual definitions in classical mathematics.  

Given the universal property that characterises $(\Vhe,\mhe)$, and since $(\VHe,\mHe)$ is a residually discrete henselian local ring, we have a unique local morphism of $\gV$-algebras, viz.~\hbox{$\varphi\colon\Vhe\to \VHe$}. Our result is the following.

\begin{theorem}\label{thm1}
  Let $(\gK,\gV)$ be a discrete valued field with henselisation $(\KHe,\VHe)$. Let $(\Vhe,\mhe)$ be the henselisation of the residually discrete local ring $(\gV,\fm)$. The unique local morphism of $\gV$-algebras $\varphi\colon\Vhe\to \VHe$ is an isomorphism.
\end{theorem}

\subsubsection*{The construction of the henselisation of a residually discrete local ring}

An elementary step in the construction of the henselisation~$\Ahe$ of a residually discrete local ring $(\gA,\fm)$ is as follows.
Let $f$ be a Hensel polynomial. 
Consider $\Ax=\Aq \gA f$, $S=1+\fm+\gen{x}$, and $\gA_f:=S^{-1}\Ax$. The ring $\gA_f$ is a residually discrete local ring with $\Rad(\gA_f)=\fm\gA_f$ and \(\frac x1\)~is a Hensel zero of~\(f\) in~$\gA_f$. The natural morphism $\gA\to\gA_f$ is local, faithfully flat, and we identify~$\gA$ with a subring of $\gA_f$. This morphism solves the universal problem for the Hensel code~\((f,0)\) in the category of $\gA$-algebras which are residually discrete local rings. 

We obtain the henselisation~$\Ahe$ as a filtered colimit of extensions of type \hbox{$\gA\to\gA_f$}. We have a filtered colimit because of the universal property satisfied by the morphism $\gA\to\gA_f$. We thus obtain~$\Ahe$ which is a residually discrete local ring of radical $\Rad(\Ahe)=\fm\Ahe$, and the morphism $\gA\to\Ahe$
yields the henselian residually discrete local ring ``generated by~$\gA$'' in the sense of the left adjoint functor to the inclusion functor in the category of residually discrete local rings.

\subsubsection*{Special polynomials}

The following lemma, directly inspired by \citet*[Lemma~5.3 and Proposition~5.4]{ALP08}, tells us that in the construction of the henselisation, we can restrict ourselves to the case where we add a special zero of a special polynomial. 
 
\smallskip %:     Lemma{lemTrick1}
\begin{lemma} \label{lemTrick1} Let $(\gA,\fm)$ be a residually discrete local ring 
and $f(X)=a_0+a_1X+\dots+a_nX^n\in\AX$ with $a_0\in\fm$ and $a_1\in\Ati$.
We define the special polynomial 
\[
\formu{  
    g(X)&=&X^n-X^{n-1}+a_{0}%\cdot
\big(a_1^{-2}a_2X^{n-2}-a_0a_1^{-3}a_3X^{n-3}+\dots+(-1)^na_0^{n-2}a_1^{-n}a_n\bigr)%=\som_{j=2}^{n}(-1)^j a_{j}a_{0}^{j-2} a_{1}^{-j} X^{n-j}\big) 
\\[.4em]
&=& X^n-X^{n-1}+a_{0}\ell(X) \quad \hbox{with } \ell(X)\in\AX
  } 
\]
and ${g_1}(X)=g(X+1)$: \(g_1\) is a Hensel polynomial with Hensel zero~\(\mu\in\fm \gA_{g_1}\) and the special zero of~\(g\) is~\(\delta=1+\mu\). The polynomial~$f$ has a zero $\gamma$ in $a_0\gA_{g_1}\subseteq \fm\gA_{g_1}\subseteq \fm\Ahe$, with~$f'(\gamma)$ a unit in $\gA_{g_1}$, and $\gamma$~is the unique zero of~$f$ in $\fm\gA_{g_1}$. If \(f\)~is monic, \(f\)~is a Hensel polynomial, \(\gA_f=\gA_{g_1}\), and \(\gamma\)~is the Hensel zero of~\(f\).
\end{lemma}
%----------- end lemma ----------------------------------- 
%
\begin{proof}
The following equality occurs in $\gA[X,1/X]$:
\[
  a_0g(X) = X^nf\Bigl(\frac{-a_0a_1^{-1}}X\Bigr).  \eqno (*)
\]
% Let $\delta=1+\mu$, where $\mu\in\fm \gA_{g_1}$ is the special zero of the special polynomial~$g$ (the Hensel code of $\delta$ is $(g,1)$). 
% The element $\mu\in\gA_{g_1}$ is the Hensel zero of code $({g_1},0)$ and
Note that $\delta\in
{\gA_{g_1}}\!\!\eti$.  
Let $\gamma=-a_0 (a_1\delta)^{-1}\in\fm\gA_{g_1}$. Applying~$(*)$ we get
$-a_0 g(\delta)=\delta^n f(\gamma)$, so $f(\gamma)=0$. Moreover $f'(\gamma)\in{\gA_{g_1}}\!\!\eti$ because $f'(0)\in\Ati$.
% and $\gamma\in\fm\gA_{g_1}=\Rad(\gA_{g_1})$.

\noindent Uniqueness is straightforward without the need to assume that $f$ be monic.
\end{proof}

\subsubsection*{Equality to~\(0\) in the henselisation of a discrete valued field}

  Consider now the  discrete valued field $(\gK,\gV)$ and let \(\gA=\gV\). Consider the natural morphism $\theta\colon\gV\!_{g_1}\to\VHe$. Suppose $(\Kac,\Vac)$ is an extension 
of $(\gK,\gV)$ to an algebraic closure.\footnote{A constructive use of the purely ideal structure $(\Kac,\Vac)$ is obtained by considering the dynamic algebraic structure defined by adding the positive diagram of $(\gA,\fm)$ to the dynamic theory of algebraically closed discrete valued fields. This application of the dynamic method is explained in more detail in \citealt{KL00}.} The advantage of considering a special polynomial is that all its zeros in~$\Kac$, with the exception of the special zero, are in $\Rad(\Vac)=\fm\Vac$.\footnote{This last equality follows from \citealt{KL00}.} %Note that \(\gV\!_{g_1}=\gV\!_f\).

Consider \(\alpha=p(\mu)\in\gV\!_{g_1}\) for $p\in\VX$ and its image \(\alpha'=p(\mu')\), where~$\mu'$ is the Hensel zero of~$g_1$ in~$\VHe$.\footnote{This is not restrictive, as we want to test only equality to~$0$.} 
\citet[beginning of the proof of proposition 2.3]{KL00} explain how to test if $\alpha'=0$.  
We calculate the characteristic polynomial $q(T)\in\VT$ of~$p(x)$ in the $\gK$-algebra 
$\Kx=\Aq \gK {g_1}$.
 The zeros of~$q$ in an algebraic closure of~$\gK$ are the $\alpha'_i=p(\mu'_i)$ where the~$\mu'_i$ are the zeros of ${g_1}$ (let $\mu'_1=\mu'$ and $\alpha'_1=\alpha'$). If~$q(0)\neq 0$, $\alpha'\neq 0$ because $q(\alpha')= 0$.\footnote{Then $\alpha\neq 0$ because $\alpha=0$ implies $\alpha'=0$. Note that here, at this stage of the proof, $\alpha\neq 0$ simply means $\lnot(\alpha=0)$, as we do not yet know whether we have a zero test in~$\Vhe$.} If $q(0)=0$ we still don't know whether $\alpha'=0$. 
 We then calculate the characteristic polynomial~$r(T)\in\VT$ of $xp(x)$ in the $\gK$-algebra~$\Kx$. 
 The $\mu'_i$ for $i\geq 2$ are units, their valuation is zero.
 The zeros of~$r$ in~$\Kac$ are the $\mu'_i\alpha'_i$. 
 The valuation of the zeros of~$r$ is therefore the same as that of the zeros of~$q$, except in general the valuation of $\alpha'=\alpha'_1$ which increases by $v(\mu')$ when $v(\alpha')$ is finite, i.e.\ when $\alpha'\neq 0$. On the other hand $\infty+v(\mu')=\infty$ and we therefore obtain that $\alpha'=0$ if, and only if, the valuation of the zeros has not changed. The valuation of the zeros of a polynomial is given by Newton's polygon method. We therefore have an explicit test for~\hbox{$\alpha'=0$}.

 The Newton polygon of a polynomial $q\in\VX$ with respect to~\(v\) also provides information about its zeros in extensions. An ``isolated slope'' of this polygon corresponds to a \textsl{$v$-isolated} zero in~$\KHe$;
\citet[Proposition~2.2]{KL00} give a precise description of this zero which ends as follows. In summary, a zero $\alpha$ corresponding to an isolated slope of a Newton polygon can always be made explicit either 
as an element of~$\gK$, or in the form 
$(a\delta+b)/(c\delta+d)$, where $\delta$ is the special zero of a special 
polynomial~$g(X)$, $a,b,c,d\in V$, $(c\delta+d)\neq 0$ and $(ad-bc)\neq 0$.

\begin{remark}
  In the following we use the fact that, since $\gV$ is integrally closed, $\gV\!_{g_1}$ has no zerodivisor \citep[Theorem~6.3]{CL2016b}.\footnote{By Theorem 6.3, the ring $\Aqj{\gV\!} f$ is normal. $\gV\!_{g_1}$ is a localisation of $\Aqj{\gV\!} f$, so that it is also normal, and as it is local it has no zerodivisor.} In the end, we deduce that $\gV\!_{g_1}$ not only has no zerodivisor but is in fact integral. In the appendix, we give two independent constructive proofs of the fact that $\gV\!_{g_1}$ is integral and integrally closed. \eoe
\end{remark}

We shall also need the following observation.
% :     Lemma{lemTrick2}
\begin{lemma}\label{lemTrick2} If $\eta\in\gK[\delta]$,
where $\delta$ is a special zero of a special polynomial,
there exists an exponent~$m$ such that $\delta^m\eta$ is a $v$-isolated zero of a polynomial $Q\in\VX$.
\end{lemma}
%----------- end lemma ----------------------------------- 
%
\begin{proof}
Result given in \citealt[proof of proposition 2.3]{KL00}.
\end{proof}

\rdb\addcontentsline{toc}{subsection}{Proof of Theorem~\ref{thm1}} \subsection*{Proof of Theorem~\ref{thm1}}

This paragraph proves the desired isomorphism. We keep the above notation.

\subsubsection*{Injectivity}

For injectivity, it is sufficient to prove it for an elementary step in the construction of~$\Vhe$ described in \citealt*{ALP08}.

In the present case with $\gA=\gV$ we obtain in particular a unique $\gV$-local morphism $\theta\colon\gV\!_{g_1}\to\VHe$. It sends the Hensel zero~$\mu$ of~$g_1$ in $\gV\!_{g_1}$ to the Hensel zero~$\mu'$ of~$g_1$ in~$\VHe$ and \(\alpha=p(\mu)\)~to~\(\alpha'=p(\mu')\).

We have to show that $\theta$ is injective, i.e.\ to show that \(\alpha'=0\implies\alpha=0\). %Let us reason here with the polynomials~$g$ and ${g_1}$ given in Lemma \ref{lemTrick1}.

If $q(0)\neq 0$ then $\alpha'\neq 0$ therefore $\lnot(\alpha= 0)$. Otherwise let's write $q(T)=T^mq_1(T)$ with $m> 0$ and $q_1(0)\neq 0$. If $\alpha'=0$, then $q_1(\alpha')\neq 0$. This implies $q_1(\alpha)\neq 0$, and as \hbox{$q(\alpha)=\alpha^mq_1(\alpha)=0$}, we get $\alpha=0$. Now we know that the natural morphism $\theta\colon\gV\!_{g_1}\to\VHe$ is injective and that $\gV\!_{g_1}$ itself is integral.   

\subsubsection*{Surjectivity}

In the following we can therefore identify any ring $\gV\!_f$ with its image in~$\Vhe$, and~$\Vhe$ with its image in~$\VHe$.
 To ensure that $\varphi$ is surjective, it suffices to show that for any Hensel zero $\mu\in\KHe$ of a polynomial~$g_1\in\VX$, any element $\eta$ of 
 \[\gW=\sotq{\eta \in\gK[\mu]}{v(\eta )\geq 0}\subseteq \VHe
 \] 
is in the image of~$\Vhe$.

Note that this implies that $\KHe$ is the field of fractions of~$\Vhe$. However, this does not tell us that $\VHe\subseteq \Vhe$. To obtain this inclusion we must use Lemma~\ref{lemTrick2}. It shows that $\delta^m\eta$ is in the image of a certain $\gV\!_u\subseteq\Vhe $. As $1/\delta$ is also in~$\Vhe$, this shows that $\eta\in\Vhe$. This completes the proof of the surjectivity of $\varphi$. 

We have completed the proof of the fact that the canonical morphism 
is an isomorphism. 
We now turn to the appendices. 

%:Appendix 1 

\rdb\addcontentsline{toc}{subsection}{Appendix 1} \subsection*{Appendix 1}

Here is another constructive proof, different from the one given in the above proof, for the fact that $\gV\!_f$ is an integrally closed ring.

\smallskip \citet*{ALN2021} consider the situation of a Hensel polynomial~$f$ for a residually discrete local ring $(\gA,\fm_\gA)$ when $\gA$ is integral. The assumption is that $(\gA,\fm_\gA)$ is an integral local ring dominated by a valuation ring~$\gV$ of the field of fractions~$\gK$ of~$\gA$ (so $ \fm_\gA=\gA\cap\fm_\gV$). We consider a step of the construction of each of the henselisations~$\Ahe$ and~$\VHe$.
We thus have a residually discrete local ring $\gA_f\subseteq \Ahe$ obtained by adding the Hensel zero~$\beta$ of~$f$ and a valuation ring~$\gW$ extension of~$\gV$ obtained by adding the Hensel zero~$\beta'$ of~$f$.

Precisely $\gW=\sotq{\zeta\in\gK[\beta']}{v(\zeta)\geq 0}$.\footnote{$\gK[\beta']$ is constructed as a discrete valued field according to a procedure specific to discrete valued fields. It is a finite separable extension of~$\gK$, but its dimension as a $\gK$-vector space is not generally known. A similar phenomenon occurs when we construct the real closure of a discrete ordered field.}
Consider the morphism $\theta_f\colon\gA_f\to \gW$ which extends the inclusion morphism $\gA\to \gV$. The paper shows that $\Ker(\theta_f)$ is a \textsl{minimal detachable prime ideal} of $\gA_f$.
We apply this result to the present situation where $\gA=\gV$ is self-dominated: the kernel $\Ker(\theta_f)$ of the natural morphism $\theta_f\colon\gV\!_f\to \gW$ is a minimal detachable prime ideal. 

\smallskip $\gV\!_{f}$ is a residually discrete local ring, localisation of~$\gV\!_{\{ f\}}$, which is normal by Theorem~6.3 because $\gV$ is a valuation domain. So $\gV\!_{f}$ is a local normal ring and therefore has no zerodivisor: this gives the result used in the previous proof. In other words the ideal $\{0\}$ of $\gV\!_{f}$ is a prime ideal (in the sense that the quotient ring has no zerodivisors). We conclude by saying that since $\Ker(\theta_f)$ is a minimal detachable prime ideal and $\{0\}$ is a prime ideal, $\{0\}$ is detachable and $\gV\!_{f}$ is integral.

%%%%%%%%%%%%%%%%%%%%%%%%%%%%%%%%%%%%%%%%%%%%%%%%%%%%%%%%%%%%%%%%%%%%
%:Appendix 2 

\rdb\addcontentsline{toc}{subsection}{Appendix 2} \subsection*{Appendix 2}

We now give a third, more direct, proof for the fact that 
the ring $\gV\!_f\subseteq \Vhe$ is integrally closed. The calculations underlying the three proofs are undoubtedly more similar than is apparent at first glance.

In fact Lemma~\ref{lemLND2} below and the analogous Theorem 6.3 in \citealt{CL2016b}
are both based on the formula of~Tate\footnote{A classical result is that if a ring~$\gA$ is a normal domain (i.e.\ integrally closed),
the same applies to~$\AX$.
A difficulty from an algorithmic point of view is that the usual proof of the previous result no longer works when we assume only that $\gA$ has no zerodivisor and not that it is integral. This was the primary motivation for the article \citealt{CL2016b}.} in Lemma~\ref{lemTate}.

%: Lemma{lemTate}
\begin{lemma}[Tate] \label{lemTate} 
Let $\gA$ be a commutative ring and $f\in \AX$ a monic polynomial of degree~$d$. 
Let us write \[f(X)=f(Y)+(X-Y)g(X,Y)\text{ with }g(X,Y)=\som_{j=0}^{d-1}g_j(Y)X^{j}\] and let us denote by~$\tr\colon\Aq\gA f\to\gA $ the trace form
of the extension $\Aq\gA f$. Then we have, for all $b\in\Aq\gA f=\Ax$,
\begin{equation}\label{eqnTate}
f'(x)b= \som_{j=0}^{d-1}\tr(g_j(x)b)\,x^{j}\text.  
\end{equation}
Note that $b,f'(x),g_j(x)\in \Aq\gA f$ and $\tr(g_j(x)b)\in \gA$.
\end{lemma}
%--------- fin lemma---------------------------------------------- 
%
\begin{proof}
See \citet[Chapter VII, corollary page~74]{Raynaud70} or Appendix~3.
\end{proof}

Remark~\ref{note1} prepares the corollary which follows.
% n
%:     Note{note1}
\begin{remark} \label{note1}
 Let $h\in \AX$ be a monic polynomial of degree~$d$ and $\Delta=uh+vh'=\pm\Res(h,h')$ be the discriminant of~$h$. Then~$h'$ is a unit in $\Aq\gA h$ if, and only if, $\Delta$ is a unit in~$\gA$ \citep[Basic elimination lemma III-7.5]{CACM}. Let us suppose that $\gA$ is an integral domain with field of fractions~$\gK$ and that $h$ is separable over~$\gK$.
We clearly have $\Aq\gA h \subseteq \Aq\gK h$ (respectively free over~$\gA$ and $\gK$ with basis $1,x,\dots,x^{d-1}$). Since~$h'$, $v$, and $\Delta$ are units in $\Aq\gK h$, they are regular in $\Aq\gA h$. We have $\Aq\gK h=\Aqj\gK h=\Aq\gK h[\frac 1 \Delta]$. Furthermore 
we have $\Aqj\gA h[\frac 1 v]=\Aq{\gA[\frac 1 \Delta]} h\subseteq \Aq\gK h$.
Thus the natural morphism $\Aq \gA h\to \Aq \gK h$ factorises into
% equation label {eqnote1}
\begin{equation} \label {eqnote1} \ndsp
\Aq \gA h \to \Aqj \gA h\to {\Aqj \gA h}[\frac 1 v]=\Aq{\gA[\frac 1 \Delta]} h\to  \Aq \gK h= \Aqj \gK h\text,
\end{equation}
where the morphisms are canonical injections which are localisation morphisms obtained by inverting regular monoids (successively ${h'}^\NN$,~$v^\NN$, and the nonzero elements of~\(\gA\)).\eoe
% end-equation
\end{remark}
%--------- fin note ------------------------------ 

%c
%:     Corollary{corlemTate}
\begin{corollary} \label{corlemTate}
Let $\gA$ be integrally closed and $\gK$ its field of fractions. Let $h$ be a separable monic polynomial of~$\AX$. Then 
$\Aqj \gA h$ is normal quasi-integral, with $\Aq \gK h=\Aqj \gK h$ as total ring of fractions. 
\end{corollary}
%--------- end corollary ------------------------------- 
%
\begin{proof}
If $d=\deg h$, $\Aq \gA h$ is free of dimension~$d$ over~$\gA$ and $\Aq \gK h=\gK\otimes_\gA \Aq \gA h$ is an étale $\gK$-algebra of dimension~$d$.  
We therefore have $\Aq \gA h\subseteq \Aq \gK h$; then $\Aqj \gA h= \Aq \gA h[1/h']\subseteq \Aq \gK h= \Aqj \gK h$.  To show that $\Aqj \gA h$ is integrally closed, we start by showing that any element~$b\in \Aq \gK h$ integral over $\Aq \gA h$ is in  
$\Aqj \gA h$. Let $b\in \Aq \gK h$ be integral over $\Aq \gA h$. As $\Aq \gA h$ is integral over~$\gA$, $b$ is integral over~$\gA$. 
We have $h' b\in \Aq \gA h$ according to Lemma~\ref{lemTate} (Equation~\eqref{eqnTate} with~$h$ in place of~$f$ and $\gK$ in place of~$\gA$):
since $b$ is integral over~$\gA$ and~$h_j(x)\in\gA$ for each~$j$, $\tr(h_j(x)b)$ is integral over~$\gA$, therefore in~$\gA$. Consequently~\hbox{$h' b\in \Aq \gA h$} and $b\in \Aqj \gA h$.

\noindent In particular, since any idempotent of $\Aq \gK h$ is integral over any subring, it belongs to $\Aqj \gA h$. 
Let $c$ be an arbitrary element of $\Aqj \gA h\subseteq \Aq \gK h$.
If $e^2=e$ and $\gen{e}=\Ann(c)$ in $\Aq \gK h$, a fortiori $e^2=e$ and $\gen{e}=\Ann(c)$ in
$\Aqj \gA h$. So $\Aqj \gA h$ is quasi-integral and 
its total ring of fractions is $\Aq \gK h$. Thus $\Aqj \gA h$ is quasi-integral and integrally closed in its total ring of fractions, so it is normal.
\end{proof}
%

%: Lemma{lemLND0}
\begin{lemma} \label{lemLND0}
Let $f\in \AX$ be a monic polynomial. The pairs of complementary idempotents of $\Aq \gA f$ are in bijective correspondence with the factorisations~$f=gh$ with~$g,h$ monic and $\gen{g,h}=\gen{1}$. For such a factorisation we obtain a natural isomorphism $\Aq \gA f\simeq \Aq \gA g\times \Aq \gA h$. 
\end{lemma}
%--------- end lemma ---------------------------------------------- 
%
\begin{proof}
If $gu+hv=1$ in $\Aq \gA f=\Ax$, the factorisation corresponds to the idempotents~$e=gu$ and $1-e=hv$. If $\gA$ is a discrete field~$\gK$, an idempotent~$e(x)$ corresponds to the element~$g(x)$, where $g(X)=\pgcd(f(X),e(X))$ in~$\KX$.  
\citet*{ALP08} give a proof in the case of a ring. We will use this lemma with a discrete field only.
\end{proof}
%
%: Lemma{lemLND01}
\begin{lemma} \label{lemLND01}
Let $\gK$ be a discrete field, $f\in \KX$ a monic polynomial, and $b\in \Aq \gK f$. We have a factorisation~$f=gh$ as in Lemma~\ref{lemLND0} with~$b$ a unit in $\Aq \gK h$ and nilpotent in $\Aq \gK g$. 
\end{lemma}
%--------- fin lemma ---------------------------------------------- 
%
\begin{proof}
The algebra $\gB:=\Aq \gK f$ is strictly finite zerodimensional. In a reduced zerodimensional ring, for any element~$b$ there is an idempotent~$e$ such that $\gen{b}=\gen{e}$, $b$ is a unit in $\gB[1/e]\simeq \aqo \gB {1-e} $ and nilpotent in $\aqo\gB e$. Apply Lemma~\ref{lemLND0}. 
\end{proof}
%

%: Lemma{lemLND1}
\begin{lemma} \label{lemLND1}
Let $\gK$ be a discrete field and $f\in \KX$ be a monic polynomial.\\
Then $\Aqj\gK f \simeq \Aqj\gK h=\Aq\gK h$
for a monic separable polynomial~$h$ which divides~$f$. 
In particular the $\gK$-algebra $\Aqj\gK f$ is étale. 
\end{lemma}
%--------- end lemma ---------------------------------------------- 
%
\begin{proof}  
Apply Lemma~\ref{lemLND01} with $b=f'(x)$. We have~$f=gh$ in~$\KX$, $f'=gh'+hg'$, $\Aq\gK f\simeq \Aq\gK g\times \Aq\gK h$.
The $\gK$-algebra $\Aqj\gK f$ is isomorphic to $\Aq \gK h$ because $f'$ is a unit in $\Aq \gK h$ and nilpotent in $\Aq \gK g$. In the $\gK$-algebra $\Aq \gK h$, $g$ and $f'=h'g$ are units, so $h'$ is also a unit and $h$ is separable over~$\gK$. Thus $\Aqj\gK h=\Aq\gK h$ is an étale $\gK$-algebra.
\end{proof}

Note that $\Aqj\gK f$ is zero when $f$ divides a power of~$f'$ in~$\KX$.

%: Lemma{lemLND2}
\begin{lemma} \label{lemLND2}
Let $\gA$ be a normal domain and $f\in \AX$
a monic polynomial. Then $\Aqj \gA f$ is 
% isomorphic to an $\gA$-algebra $\Aq \gA h$ for a polynomial separable of $\AX$, i.e.\
a normal quasi-integral ring. 
\end{lemma}
%--------- end lemma ---------------------------------------------- 
%
\begin{proof} %We assume that $\gA$ is nontrivial.
Let $\gK=\Frac\gA$% and $\varphi\colon\Aqj \gA f \to \Aqj \gK f$ be the natural morphism
.  
We apply Lemma~\ref{lemLND01} with $b=f'$ and Lemma~\ref{lemLND1}, we write {$f=gh$} with monic polynomials~$g$ and~$h$ in~$\KX$. We have $g,h\in \AX$ because $\gA$ is integrally closed.
Moreover $g$ monic divides a power of~$f'$ in~$\KX$ so in~$\AX$ (the division takes place entirely in~$\AX$). So $g$ is a unit in $\Aqj\gA f$, which implies~$h=0$ in $\Aqj\gA f$ because $gh=f=0$. 
We therefore obtain 
\[\ndsp
\Aqj\gA f=\Aq\gA{gh}[\frac 1{f'g}]\simeq \Aq\gA{gh}[\frac 1{g}][\frac 1{h'}]\simeq \Aq\gA{h}[\frac 1{g}][\frac 1{h'}]=\Aqj\gA h[1/g]\text.
\]
Relative to the natural inclusion $\Aq \gA h\to \Aq \gK h$, the elements~$h'$, $f'$, and~$g$, units in $\Aq \gK h$, are regular in $\Aq \gA h$. We therefore obtain the canonical injective morphisms of localisation at regular monoids shown below:
\[\Aq \gA h \to\Aqj \gA h \to \Aqj\gA f%\simeq \Aq \gA h[1/f']%\simeq \Aqj \gA h[1/g]
\to   \Aq\gK h\simeq \Aqj\gK f.
\]

\noindent 
We conclude that $\Aqj\gA h$ is normal quasi-integral, and therefore also its localisation~$\Aqj\gA f\simeq\Aqj \gA h[1/g]$.
\end{proof}

%:     Lemma{lemLND3}
\begin{lemma} \label{lemLND3}
Let $\gA$ be an integrally closed residually discrete local ring and $f\in \AX$
a monic polynomial. Then $\gA_f $ is an integrally closed residually discrete local ring. 
\end{lemma}
%--------- end lemma ---------------------------------------------- %
%
\begin{proof}
We know that $\gA_f $ is a residually discrete local ring.  As a localisation of~$\Aqj \gA f$, it is a normal quasi-integral ring.  And every local ring is connected.
\end{proof}
%

%%%%%%%%%%%%%%%%%%%%%%%%%%%%%%%%%%%%%%%%%%%%%%%%%%%%%%%%%%%%%%%%%%%%
%:Appendix 3 

\rdb\addcontentsline{toc}{subsection}{Appendix 3} \subsection*{Appendix 3}

In this appendix, we recall a proof of Tate's lemma given in
\citealt{CL2016b} which is simpler than the one of \citet[Chapter~VII, §1]{Raynaud70}.
\begin{proof}[Proof of Lemma~{\ref{lemTate}}] We have $\Aq\gA f=\Ax=\aqo\AX f$.
Let $\gB=\gA[x_1,\dots,x_d]$ be the universal decomposition algebra of~$\gA$ for~$f$, where $f(X)=\prod_{i=1}^d(X-x_i)$ \citep[Definition III-4.1]{CACM}. Each $\gA[x_i]\subseteq \gB$ is isomorphic to $\gA[x]$ and we can take $x=x_1$.  
We write $b=h(x) \in \Aq\gA f$. For all $\ell(x)\in \Aq\gA f$ we have $\tr (\ell(x))=\som_{i=1}^d\ell(x_i)\in\gA$.\footnote{See for example \citet[Lemma III-5.12]{CACM}.} We have $g(Y,Y)=f'(Y)$ and $g(x_1,x_i)=g(x_i,x_1)=0$ if $i\neq 1$: in fact $g(X,x_1)=\prod_{i=2}^d(X-x_i)$ because $(X-x_1)g(X,x_1)=f(X)-f(x_1)=f(X)=\prod_{i=1}^d(X-x_i)$.
So 
\[f'(x_1)b=g(x_1,x_1)h(x_1)=\som_{j=0}^{d-1} g_j(x_1)h(x_1)x_1^j\]
and for $i\neq 1$ 
\[0=g(x_1,x_i)h(x_i)=\som_{j=0}^{d-1} g_j(x_i)h(x_i)x_1^j\text.\]
Summation gives the Tate formula announced.
\end{proof}
%

%\hum{In \cite{CL2016b}, we do not say that $f$ is separable but we assume that $\gA$ is normal. This assumption is used to show that $\tr(b)$ has a value in $a\gA$ if $b$ is integral over $a\gA$, but this intermediate result disappears in the calculation that follows, reproduced above.
%Moreover, the result $g(x_1,x_j)=0$ if $jneq 1$ seems absolutely necessary~$f$ separable.}.
%%%%%%%%%%%%%%%%%%%%%%%%%%%%%%%%%%%%%%%%%%%%%%%%%%%%%%%%%%%%%%%%%%%%
%:References \newpage
% \small
\bibliographystyle{plainnat}
\bibliography{Note2Hbib}

\end{document}